\documentclass{amsart}

\usepackage{amsmath}
\usepackage{amsthm}
\usepackage{amssymb}
\usepackage{stmaryrd}
\usepackage{mathrsfs}
\usepackage[all]{xy}
\usepackage{url}

\theoremstyle{plain} 
\newtheorem{thm}{Theorem}[section]
\newtheorem{cor}[thm]{Corollary}
\newtheorem{prop}[thm]{Proposition}
\newtheorem{lem}[thm]{Lemma}
\theoremstyle{remark} 

\newtheorem{rem}[thm]{Remark}
\theoremstyle{definition} 
\newtheorem{defn}[thm]{Definition}

\newtheorem*{mainthm}{Theorem \ref{t:main}}

\newcommand{\cF}{\ensuremath{\mathscr F}}

\newcommand{\cO}{\ensuremath{\mathscr O}}

\newcommand{\cHom}{\ensuremath{\mathscr Hom}}

\newcommand{\scHom}{\ensuremath{s\mathscr{H}om}}
\newcommand{\set}[1]{\ensuremath{\left\lbrace #1\right\rbrace}}

\newcommand{\inv}{\ensuremath{{-1}}}
\newcommand{\Q}{\ensuremath{\mathbb{Q}}}

\newcommand{\C}{\ensuremath{\mathbb{C}}}
\newcommand{\Z}{\ensuremath{\mathbb{Z}}}

\newcommand{\del}{\ensuremath{\partial}}

\DeclareMathOperator{\HH}{\bf{H}}
\DeclareMathOperator{\Tot}{Tot}
\DeclareMathOperator{\Hom}{Hom}
\DeclareMathOperator{\RHom}{RHom}

\DeclareMathOperator{\Ext}{Ext}

\DeclareMathOperator{\Ab}{Ab}
\DeclareMathOperator{\Sh}{Sh}
\DeclareMathOperator{\Pre}{Pre}
\DeclareMathOperator{\Ch}{Ch}

\DeclareMathOperator{\id}{id}
\DeclareMathOperator{\sk}{sk}
\DeclareMathOperator{\cosk}{cosk}

\DeclareMathOperator{\codim}{codim}
\DeclareMathOperator{\Sch}{Sch}
\DeclareMathOperator{\Sm}{Sm}
\DeclareMathOperator{\Zar}{Zar}
\DeclareMathOperator{\trace}{trace}
\newcommand{\dt}{\ensuremath{\bullet}}

\newcommand{\blank}{\ensuremath{\underline{\quad}}}
\newcommand{\uDelta}{\ensuremath{\underline{\Delta}}}
\newcommand{\udDelta}{\ensuremath{\underline{\del\Delta}}}
\newcommand{\dDelta}{\ensuremath{\del\Delta}}

\begin{document}

\title{Local acyclic fibrations and the De Rham complex}
\author{Ben Lee}
\date{\today}

\begin{abstract}
  We reinterpret algebraic de Rham cohomology for a possibly singular
  complex variety $X$ as sheaf cohomology in the site of smooth
  schemes over $X$ with Voevodsky's $h$-topology.  Our results extend
  to the algebraic de Rham complex as well.  Our main technique is to
  extend \v{C}ech cohomology of hypercovers to arbitrary local acyclic
  fibrations of simplicial presheaves.
\end{abstract}

\maketitle

\section{INTRODUCTION} 

Let $X$ be a separated scheme finite type over the complex numbers
$\C$.  Following Deligne, Du Bois (\cite{MR613848}) constructs the
algebraic de Rham complex of $X$ \[ \underline{\Omega^\dt_{X/\C}} :=
Re_*\Omega^\dt_{X_\dt/\C}
\] by a choice of a smooth proper hypercover.  It is well-defined in
the filtered derived category.  Morally this \v{C}ech complex should
be a derived direct image from some topos to the Zariski site; showing
this is the aim of this paper.

The choice of topos appears to be a delicate matter.  Using the
topology of ``universal cohomological descent'' (which we abbreviate
``$ucd$'') on proper and smooth schemes turns out to be technically
inconvenient.  We use instead Voevodsky's $h$-topology
\cite{MR1764199} on possibly open schemes.  Denote by $\Sm_h/X$ the
category of smooth separated schemes finite type over $X$, equipped
with the $h$-topology.  We show the presheaf $\Omega^q$ is a sheaf on
$\Sm_h / X$.  There is a direct image $\gamma_*$ from sheaves on
$\Sm_h/X$ to sheaves on the small Zariski site $X_{\Zar}$.

Unfortunately we cannot directly apply Verdier's work on \v{C}ech
cohomology of hypercovers.  Comparing \v{C}ech and derived functor
cohomology in this situation requires finite fiber products which
don't exist in $\Sm$.  However the standard comparison would show
that \[ \underline{\Omega^\dt_{X/\C}} \simeq R\gamma_*\Omega^\dt \]
and thus \[ H_{dR}^i(X) \simeq \HH_h^i(X, \Omega^\dt) =
\HH_{\Zar}^i(X, R\gamma_*\Omega^\dt) \] giving our main result.
(By GAGA \cite{MR2017446} and results of \cite{MR0199194}, this would
be isomorphic to its analytic counterpart.)

According to Jardine (\cite{MR906403}), hypercovers are just
(semi-)representable local acyclic fibrations.  Keeping this in mind,
we generalize Verdier's work on \v{C}ech cohomology to arbitrary local
acyclic fibrations of simplicial presheaves.  The precise statement
proved in the first section is
\begin{mainthm}
  Let $X_\dt$ be a simplicial presheaf, and $hD(X_\dt)$ the homotopy
  category of local acyclic fibrations $K_\dt \to X_\dt$.  Then for a
  bounded below complex of sheaves of abelian groups $\cF^\dt$ with
  the filtration b\^{e}te there is an isomorphism
  \[ \varinjlim_{K_\dt \in hD(X_\dt)} H^p(\Tot\Hom(\Z{K^\dt}, \cF^\dt))
  \simeq \Ext^p(\Z{X^\dt}, \cF^\dt) \] and there is a
  filtered quasi-isomorphism of ind-objects in the derived category
  \[ \underset{K_\dt \in hD(X_\dt)}{"\varinjlim"} \Tot\Hom(\Z{K^\dt},
  \cF^\dt) \simeq R\Hom(\Z{X^\dt}, \cF^\dt). \]
\end{mainthm} 
Note the lack of hypotheses on fiber products in the underlying topos.

In practice one usually wants to restrict to local acyclic fibrations
which satisfy some representability hypothesis.  Define a
semi-representable presheaf to be a presheaf that is isomorphic to a
coproduct of representable presheaves; that one can restrict to
semi-representable presheaves is an easy corollary of the above
theorem. To satisfy stronger hypotheses than semi-representability
seems to require something from the underlying topos -- in our case we
use the inclusion $\Sm_h \subset \Sch_h$.

Section three is occupied with ``topological'' matters.  Using Du
Bois' results requires a comparison of the $ucd$- and $h$- topologies:
after some preliminaries, we show every $h$-covering is a
$ucd$-covering.  We do not know of an example of a $ucd$-covering tha
is not an $h$-covering.  Finally we show one can actually compute
using representable presheaves in $\Sm_h$.

In section four we apply our work to the algebraic de Rham complex.
Key in applying Du Bois' results is Theorem \ref{t:beilinson}, which
compares $h$-hypercovers to Zariski hypercovers.  This result comes
from a generous suggestion of Alexander Beilinson.  Also in this
section is the proof that $\Omega^q$ is a sheaf in the $h$-topology.
These results with the \v{C}ech theory yield the main theorem.

This paper is based on my dissertation, and I owe thanks to the many
people who helped me.  Everything here has benefited from the guiding
hand of my advisor, Madhav Nori, to whom I give my sincerest thanks.
Alexander Beilinson has also provided invaluable help and advice.  I
would also like to thank Andrew Blumberg and Minhea Popa for
stimulating mathematical discussions.

\section{A Generalized Verdier Theorem}

\subsection{Local acyclic fibrations}

Let $C$ be a site, $\Pre C$ the category of presheaves of sets on $C$,
$\Sh C$ the category of sheaves of sets on $C$, and $s\Pre C, s\Sh C$
the categories of simplicial presheaves and sheaves.  Note that,
unlike \cite[ex V 7.3.0]{MR0354653}, we do not assume the existence of
products and finite fiber products in our site $C$.  Let $e$ be the
terminal object of $\Sh C$.  For a presheaf $K$, let $\Z{K}$ denote
the associated sheaf of free abelian groups; for a simplicial presheaf
$K_\dt$, let $\Z{K^\dt}$ denote the associate negative cochain complex
of sheaves of free abelian groups; define $\Z := \Z{e}$ the sheaf of
free abelian groups associated to the terminal object $e$.

\begin{defn}[cf. \cite{MR906403,MR2066498,MR2034012}] \quad

  \begin{enumerate}
  \item Let $f:L_\dt \to K_\dt$ be a morphism of simplicial
    presheaves.  $f$ is called a \emph{local acyclic fibration} if,
    for every $U \in C$, integer $k \ge 0$ and diagram
    \[ \xymatrix{ \dDelta^k \ar@{^(->}[d] \ar[r] & L_\dt(U) \ar[d]_{f(U)} \\
      \Delta^k \ar[r] & K_\dt(U) } \] there is a refinement (a
    covering sieve) $R$ of $U$ so that for every $V \to U \in R$ there
    is a lift
    \[ \xymatrix{ \dDelta^k \ar@{^(->}[d] \ar[r] & L_\dt(U)
      \ar[d]_{f(U)} \ar[r] & L_\dt(V) \ar[d]_{f(V)} \\
      \Delta^k \ar[r] \ar@{-->}[rru] & K_\dt(U) \ar[r] & K_\dt(V) }
    \] indicated by the dashed arrow.  We say $f$ satisfies the
    \emph{local right lifting property} for the inclusion $\dDelta^k
    \to \Delta^k$.
  \item For a presheaf $M \in \Pre C$, let $M_\dt$ be the constant
    simplicial presheaf associated to $M$.  We abuse notation and call
    the augmented simplicial presheaf $K_\dt \to M$ a local acyclic
    fibration if the morphism of simplicial presheaves $K_\dt \to
    M_\dt$ is a local acyclic fibration.
  \item Recall a simplicial presheaf is \emph{semi-representable} if
    its components are isomorphic to coproducts of representable
    presheaves.  A local acyclic fibration $L_\dt \to K_\dt$ is a
    \emph{hypercover} if both $L_\dt$ and $K_\dt$ are
    semi-representable.
  \end{enumerate} 
\end{defn}

Compare the following with \cite[ex V Lemma 7.3.6]{MR0354653}:

\begin{lem}\label{l:cofib} A morphism $f:L_\dt \to K_\dt$ is a local
  acyclic fibration if and only if for every $P_\dt \hookrightarrow
  Q_\dt$ an inclusion of constant simplicial sets with only finitely
  many non-degenerate simplices, we can locally lift diagrams
  \[ \xymatrix{ P_\dt \ar@{^(->}[d] \ar[r] & L_\dt(U)
    \ar[d]_{f(U)} \ar[r] & L_\dt(V) \ar[d]_{f(V)} \\
    Q_\dt \ar[r] \ar@{-->}[rru] & K_\dt(U) \ar[r] & K_\dt(V) }
  \] 
\end{lem}

\begin{proof} Induction on the definition.
\end{proof}

Recall a morphism of presheaves $F \to G$ is a \emph{covering
  morphism} if the associated morphism of sheaves is an epimorphism
(see \cite[II.5.2]{MR0354652}.)

\begin{rem}
  For a morphism $f:L_\dt \to K_\dt$, Verdier uses the following
  equivalent definition of local acyclic fibration (which he calls
  ``special''):
  \begin{enumerate}
  \item For each integer $k \ge 0$, the morphism $\phi_k$ in the
    diagram is a covering morphism:
    \[ \xymatrix{ L_k \ar[rr]^{f_k} \ar[dd]
      \ar[dr]^{\phi_k} & & K_k \ar[dd] \\ & P_k \ar[ru]
      \ar[dl] & \\ (\cosk_{k-1} L)_k \ar[rr]_{(\cosk_{k-1} f)_k} & &
      (\cosk_{k-1} K)_k } \] The vertical arrows are the coskeleton
    adjunction maps and $P_k$ is the fiber product of $K_k$
    and $(\cosk_{k-1} L)_k$ by the arrows in the diagram.
  \item The morphism $f_0:L_0 \to K_0$ is a covering morphism.
  \end{enumerate}
\end{rem}

For a simplicial set $S_\dt$, let $\underline{S}_\dt$ denote the
associated constant simplicial presheaf.

\begin{prop}[{\cite[Proposition 7.2]{MR2066498}}]\label{p:lac} 
  Let $f:L_\dt \to K_\dt$ be a morphism of simplicial presheaves.
  Then the following are equivalent: 
  \begin{enumerate}
  \item $f$ is a local acyclic fibration.
  \item For every integer $k \ge 0$, the morphism 
    \[ \cHom(\uDelta^k, L_\dt) \to \cHom(\udDelta^k, L_\dt)
    \times_{\cHom(\udDelta^k, K_\dt)} \cHom(\uDelta^k, K_\dt)
    \] induced by the inclusion $\del \Delta^k \to \Delta^k$
    and $f$ is a covering morphism.
  \item $f$ is special in the sense of Verdier.
  \end{enumerate}
\end{prop}

\begin{proof}
  $1 \Leftrightarrow 2$ is by definition.  To show $2 \Leftrightarrow
  3$, apply the isomorphisms $X_k = \Hom(\Delta^k, X_\dt)$ and the
  coskeleton-skelton adjunction to the covering condition
  \[ L_k \to (\cosk_{k-1}L)_k \times_{(\cosk_{k-1}K)_k} K_k \] noting
  that $\sk_{k-1}\uDelta^k = \udDelta^k$.
\end{proof}

We recall the following basic results.

\begin{prop}[{\cite[Proposition 2.9]{MR906403}}]\label{p:jardine}
  If $f:L_\dt \to K_\dt$ is a local acyclic fibration of simplicial
  presheaves, then the induced map $\Z{L^\dt} \to \Z{K^\dt}$ is a
  quasi-isomorphism of complexes of sheaves.
\end{prop}

\begin{lem}[{\cite[ex V Lemma 7.3.4]{MR0354653}}]\label{l:lac} \quad

  \begin{enumerate}
  \item The composition of two local acyclic fibrations is a local
    acyclic fibration.
  \item Local acyclic fibrations are preserved under base change.
  \item Suppose $K_\dt$ is a generalized hypercover, $f:L_\dt \to
    M_\dt$ a local acyclic fibration, and $K_\dt \to M_\dt$ a
    morphism.  Then the fiber product $L_\dt \times_{M_\dt} K_\dt$ is
    a generalized hypercover.
  \end{enumerate}
\end{lem}

\subsection{Computing $\Ext$}

Before proving our main theorem we require the following technical lemma.

\begin{lem}[Lemma on computing $\Ext$]\label{l:ext} Let $C$ be an
  abelian category with enough injectives, $X^\dt \in \Ch^-(C)$ a
  fixed negative cochain complex, and $G^\dt \in \Ch^+(C)$ a fixed
  positive cochain complex.  Suppose $D \subset \Ch^-(C) / X^\dt$ is a
  subcategory of the category of negative cochain complexes of $C$
  over $X^\dt$ with the following properties:
  \begin{enumerate}
  \item The homotopy category $hD$ (morphisms up to chain
    homotopy) is cofiltered.
  \item For every complex $K^\dt \in D$, object $M \in C$, and
    epimorphism $u: M \to K^n$, there is a complex $L^\dt \in D$ and a
    morphism $f:L^\dt \to K^\dt$ whose degree $n$ part factors as \[
    f^n:L^n \to M \overset{u}{\to} K^n. \]
  \item Every $K^\dt \in D$ has structure morphism $K^\dt \to X^\dt$ a
    quasi-isomorphism.
  \end{enumerate} Then there is an isomorphism of functors
  \[ \varinjlim_{K^\dt \in hD} H^p(\Tot\Hom_C(K^\dt, G^\dt)) \simeq
  \Ext^p_C(X^\dt, G^\dt). \] Here $\Ext$ is hyper-$\Ext$.
\end{lem}

Some explanation:
\begin{enumerate}
\item We first work with the case when $G^\dt = G$ a single object
  concentrated in degree zero.  We compute $\RHom(K^\dt, G)$ by taking
  an injective resolution $I^\dt$ of $G$, yielding a first quadrant
  double complex $\Hom_C(K^\dt, I^\dt)$ (giving $\Hom_C(K^a, I^b)$
  bidegree $(-a, b)$) which has total complex $\Tot\Hom_C(K^\dt,
  I^\dt)$.  This complex has a decreasing filtration by columns
  \[ F^l \Tot^m\Hom_C(K^\dt, I^\dt) = \bigoplus_{\substack{-a+b = m \\ a \le l}}
  \Hom_C(K^a, I^b). \] We get a first quadrant convergent spectral
  sequence
  \begin{align*} E_1^{p,q} &= H^q(\Hom_C(K^p, I^\dt)) = \Ext^q_C(K^p,
    G) \\ &\Rightarrow H^{p+q}\RHom(K^\dt, G) =
    \Ext^{p+q}_C(X^\dt, G).
  \end{align*}
\item Since only $hD$ is cofiltered, it does not make sense to take
  the filtered colimit of the $E_1$ terms over $hD$.  However, since
  the $E_2$ terms are the horizontal cohomology of the $E_1$ terms and
  chain homotopic maps induce the same map on cohomology, we can take
  the filtered colimit of the $E_2$ terms over $hD$.  This yields a
  limit spectral sequence
  \[ E_2^{p,q} = \varinjlim_{K \in hD} H^p(\Ext^q_C(K^\dt, G))
  \Rightarrow \Ext^{p+q}_C(X^\dt, G). \] The objects on the left
  hand side are cohomologies of the complexes of $\Ext^q_C$ by varying
  the $K^\dt$.  The contention of the theorem is that the terms with
  $q > 0$ vanish in the limit, collapsing the spectral sequence at the
  $E_2$ page, yielding an isomorphism
  \[ \varinjlim_{K \in hD} H^p(\Hom_C(K^\dt, G)) = \Ext^p_C(X^\dt,
  G). \]
\item Using property 2 of $D$, the remarks show it is enough to prove
  the following well-known lemma.

  \begin{lem} Let $C$ be an abelian category with enough injectives.
    For any $K, A \in C$, any $q > 0$ and any extension class
    \[ \gamma \in \Ext^q_C(K, A) \] there is an epimorphism $f:M \to
    K$ so that $f^*(\gamma) = 0$ in $\Ext^q_C(M, A)$.
  \end{lem}
  \begin{proof} For a fixed $q > 0$ choose a truncated injective
    resolution
    \[ 0 \to A \to I^0 \to \cdots \to I^{q-1} \to J \to 0 \] where $J$
    is the cokernel of $I^{q-2} \to I^{q-1}$.  Applying the functor
    $\Hom(K, \blank)$ yields the complex
    \[ \Hom(K, I^0) \to \cdots \to \Hom(K, I^{q-1}) \to \Hom(K, J) \to
    \Ext^q(K, A) \to 0. \] Lift $\gamma$ to a homomorphism $\sigma:K
    \to J$ in $\Hom(K, J)$.  Form the fiber product $I^{q-1} \times_J
    K$ using $\sigma$.  The natural projection map $f:M = I^{q-1}
    \times_K K \to K$ yields a map of complexes
    \[ \xymatrix{ \cdots \ar[r] & \Hom(K, I^{q-1}) \ar[r]\ar[d] &
      \Hom(K, J) \ar[r] \ar[d] & \Ext^q(K, A) \ar[r] \ar[d] & 0 \\
      \cdots \ar[r] & \Hom(M, I^{q-1}) \ar[r] & \Hom(M, J) \ar[r] &
      \Ext^q(M, A) \ar[r] & 0. } \] By construction, $f^*(\sigma) \in
    \Hom(M, J)$ is $u \circ p_2$ in the cartesian square
    \[ \xymatrix{ I^{q-1} \times_K M \ar[r]^{p_2} \ar[d]_{p_1} & M
      \ar[d]^u \\ I^{q-1} \ar[r]^c & K }. \] But $u \circ p_2 = c
    \circ p_1$ is the image of $p_1 \in \Hom(M, I^{q-1})$.  Hence
    $f^*(\sigma)$ is a coboundary and so $f^*(\gamma)$ is zero.
  \end{proof}

\item Let $f:I \to D^+(C)$ be a filtered system in the derived
  category of $C$.  The associated ind-object is denoted by
  \[ \underset{M^\dt \in f(I)}{"\varinjlim"} M^\dt. \] We define the
  cohomology of this ind-object by the equation
  \[ H^k\left(\underset{M^\dt \in f(I)}{"\varinjlim"} M^\dt\right) :=
  \varinjlim_{M^\dt \in f(I)} H^k(M^\dt). \] We note that, in the case
  where the ind-object is representable, this agrees with the
  cohomology of the limit object since $H^k(\varinjlim M^\dt) =
  \varinjlim H^k(M^\dt)$, cf.  \cite[1.12.7]{MR1299726}, using the
  model of the derived category via injectives (as in
  \cite[III.5.22]{MR1950475}.)  We say a map of ind-objects is a
  quasi-isomorphism if it induces an isomorphism on cohomology.

  \begin{cor} Suppose $D \subset \Ch^-(C)$ satisfies the hypotheses of
    Lemma \ref{l:ext}.  Then there is a natural quasi-isomorphism
    \[ \underset{K^\dt \in hD}{"\varinjlim"} \Hom_C(K^\dt, G) \simeq
    R\Hom_C(X^\dt, G). \]
  \end{cor}

\item The results extend to complexes concentrated in a single
  non-zero degree, by reindexing.

\item Now let $G^\dt \in \Ch^b(C)$ be a finite complex.  By the
  corollary we see that
  \[ \underset{K^\dt \in hD}{"\varinjlim"} \Hom_C(K^\dt, \blank) \]
  takes short exact sequences to exact triangles.  If $G^\dt$ is a
  bounded complex, it has a finite truncation filtration with
  subquotients complexes concentrated in a single degree.  This gives
  the result for finite complexes.

\item  For a bounded below complex $G^\dt$, we note that
  \[ \varinjlim_n R\Hom(K^\dt, G^{\le n}) = R\Hom(K^\dt, \varinjlim_n
  G^{\le n}) = R\Hom(K^\dt, G^\dt) \] if $K^\dt$ is a bounded above
  complex: $R^i\Hom(K^\dt, G^\dt) = R^i\Hom(K^\dt, G^{\le n})$ for
  some $n$ sufficiently large, since the overlap between $K^\dt$ and
  $G^\dt[i]$ is finite.  Likewise \[ \varinjlim_n \underset{K^\dt \in
    hD}{"\varinjlim"} \Hom_C(K^\dt, G^{\le n}) = \underset{K^\dt \in
    hD}{"\varinjlim"} \Hom_C(K^\dt, \varinjlim_n G^{\le n}). \] This
  gives the result for bounded below complexes, and thus completes the
  proof of Lemma \ref{l:ext}.
\end{enumerate}

\begin{cor}\label{c:rhom} Suppose $D \subset \Ch^-(C)$ satisfies the
  hypotheses of Lemma \ref{l:ext}, and $G^\dt \in \Ch^+(C)$ is a fixed
  bounded below complex.  Then there is a natural filtered
  quasi-isomorphism
  \[ \underset{K^\dt \in hD}{"\varinjlim"} \Tot\Hom_C(K^\dt, G^\dt)
  \simeq R\Hom_C(X^\dt, G^\dt) \] where on each side the the
  filtration arises from the filtration b\^{e}te on $G^\dt$.
\end{cor}

\subsection{Main theorem}

Recall $C$ is a site, possibly without finite products and fiber
products.

\begin{defn}
  For a fixed simplicial presheaf $X_\dt \in s\Pre C$, let $D(X_\dt)$
  denote the subcategory of $s\Pre C / X_\dt$ of local acyclic
  fibrations $K_\dt \to X_\dt$.  

  For any category of simplicial objects $E$, write $hE$ to be the same
  category with morphisms up to simplicial homotopy.  In general this
  is not an equivalence relation, we use the relation generated by
  simplicial homotopy.  
\end{defn}

\begin{prop}[cf. {\cite[ex V Theorem 7.3.2]{MR0354653}}]\label{p:verd}
  \quad
  
  Fix a simplicial presheaf $X_\dt \in s\Pre C$.
  \begin{enumerate}
  \item The homotopy category $hD(X_\dt)$ is cofiltered.
  \item For every $K_\dt \in hD(X_\dt)$, object $M \in C$, and
    covering morphism $u:M \to K_n$, there is an object $L_\dt \in
    hD(X_\dt)$ and a morphism $f:L_\dt \to K_\dt$ whose degree $n$
    part factors as \[ f_n:L_n \to M \overset{u}{\to} K_n.\]
  \item For every $K_\dt \in D(X_\dt)$, the structure morphism $K_\dt
    \to X_\dt$ induces a quasi-isomorphism \[ \Z{K^\dt} \to
    \Z{X^\dt}. \]
  \end{enumerate}
\end{prop}

\begin{proof} \quad
  
  \begin{itemize}
  \item{Proof of part 3 of Proposition \ref{p:verd}}:
  
    This is just Proposition \ref{p:jardine}.

  \item{Proof of part 2 of Proposition \ref{p:verd}}:

    The following is mostly unchanged from Verdier's original.

    Let $j_{n*}$ the right adjoint of ``taking the degree $n$
    component.''  I claim $j_{n*}$ takes covering morphisms to local
    acyclic fibrations.  Let $f:A \to B$ be a covering morphism of
    presheaves.  Then we must check, for an open $U \in C$, that we
    can locally lift a diagonal in a diagram
    \[ \xymatrix{ \dDelta^k_n \ar[r] \ar@{^(->}[d] & A(U) \ar[d] \\
      \Delta^k_n \ar[r] & B(U). } \] But since $A \to B$ is a
    covering, it is a surjection after a refinement $V$ of $U$, so we
    can always lift $\Delta^k_n \to A(V)$.

    To prove part 2, form the cartesian diagram
    \[ \xymatrix{ L_\dt \ar[r]\ar[d] & j_{n*}M \ar[d] \\ K_\dt \ar[r]
      \ar[d] & j_{n*}j_n^*K_\dt = j_{n*}K_n \\ X_\dt } \] where the
    right vertical arrow is given by functoriality and the bottom
    horizontal arrow is given by adjunction.  The right vertical arrow
    is a local acyclic fibration by the above remark.  By Lemma
    \ref{l:lac} $L_\dt \to K_\dt \to X_\dt$ is a local acyclic
    fibration, and $L_n \to K_n$ factors as $L_n \to M \to K_n$.

  \item{Proof of part 1 of Proposition \ref{p:verd}}: \footnote{We
      warn the reader that there is a small, inconsequential error in
      Verdier's original.}

    Suppose we are given a diagram
    \[ \xymatrix { & A_\dt \ar[d] \\ B_\dt \ar[r] & K_\dt } \] in
    $D(X_\dt)$.  Set $L_\dt = A_\dt \times_{X_\dt} B_\dt$ which exists
    in $s\Pre C$.  Lemma \ref{l:lac} shows the canonical map $L_\dt
    \to X_\dt$ is a local acyclic fibration, so it is in $D(X_\dt)$.
    This gives a possibly non-commutative diagram
    \[ \xymatrix{ L_\dt \ar[rr] \ar[dd] & & A_\dt \ar[dl] \ar[dd] \\
      & K_\dt \ar[dr] \\ B_\dt \ar[rr] \ar[ru] & & X_\dt. } \] Hence
    we have two maps $L_\dt \rightrightarrows K_\dt$ which we wish to
    equalize up to homotopy.  Thus to prove $hD(X_\dt)$ is cofiltered,
    it is enough to show that for every pair of morphisms in
    $D(X_\dt)$
    \[ L_\dt \underset{u_1}{\overset{u_0}{\rightrightarrows}} K_\dt \]
    there is a morphism $v: M_\dt \to L_\dt$ in $D(X_\dt)$ so that the
    two morphisms $u_0v$ and $u_1v$ are homotopic, i.e. there are
    commutative diagrams
    \[ \xymatrix{ M_\dt \ar[r]^{e_i} \ar[d]_v & M_\dt \times \uDelta^1
      \ar[d]^w \\ L_\dt \ar[r]_{u_i} & K_\dt } \] for $i = 0, 1$,
    where the $e_i$ are the standard inclusions, and $w$ is the
    homotopy.

    The set of such diagrams for fixed $M_\dt$ and $L_\dt
    \underset{u_1}{\overset{u_0}{\rightrightarrows}} K_\dt$ is given
    by
    \[ \Hom(M_\dt \times \uDelta^1, K_\dt) \times_{\Hom(M_\dt, K_\dt
      \times K_\dt)} \Hom(M_\dt, L_\dt) \] where the map from
    $\Hom(M_\dt, L_\dt)$ to $\Hom(M_\dt, K_\dt \times K_\dt)$ is
    induced from $u_1 \times u_2$, and the map $\Hom(M_\dt \times
    \uDelta^1, K_\dt)$ to $\Hom(M_\dt, K_\dt \times K_\dt)$ is induced
    from $e_0 \times e_1$.

    The functor
    \[ \Hom(\blank \times \uDelta^1, K_\dt) \times_{\Hom(\blank, K_\dt
      \times K_\dt)} \Hom(\blank, L_\dt) \] is equal to
    \begin{align*}
      \Hom(\blank, \scHom(\uDelta^1, K_\dt)) \times_{\Hom(\blank,
        K_\dt \times K_\dt)} \Hom(\blank, L_\dt) \\= \Hom(\blank,
      \scHom(\uDelta^1, K_\dt) \times_{K_\dt \times K_\dt} L_\dt)
    \end{align*}
    and so is representable.  Call this representing object $F_\dt$.
    We must show that $F_\dt \to X_\dt$ is a local acyclic fibration.
    $F$ is the pullback in the square in the diagram \[ \xymatrix{
      F_\dt \ar[rr] \ar[d] & & \scHom(\uDelta^1, K_\dt) \ar[d] \\
      L_\dt \ar[r]^d & L_\dt
      \times L_\dt \ar[r]^{u_0 \times u_1} & K_\dt \times K_\dt \ar[dr] \\
      & & & X_\dt \times X_\dt \ar@<.3ex>[dr]^{p_1} \ar@<-.3ex>[dr]_{p_2} \\
      & & & & X_\dt } \] where $d$ is the diagonal.  Note all maps to
    $X_\dt$ are the same, and \[ K_\dt \times K_\dt =
    \scHom(\udDelta^1, K_\dt).\]

    Thus to lift a diagram
    \[ \xymatrix{ \dDelta^k \ar[r] \ar@{^(->}[d] & F_\dt(U) \ar[d] \\
      \Delta^k \ar[r] & X_\dt(U) } \] we have to lift from $X_\dt$ to
    $L_\dt, K_\dt \times K_\dt$ and $\scHom(\uDelta^1, K_\dt)$ with
    the following compatibility condition: the lift $\Delta^k \to
    L_\dt(V)$ yields by composition with the diagonal a lift $\Delta^k
    \to L_\dt(V) \times L_\dt(V)$, or a map $\Delta^k \times \dDelta^1
    \to L_\dt(V)$.  By composition with the map $u_0 \times u_1$ we
    get a lift $\Delta^k \times \dDelta^1 \to K_\dt(V)$.  Meanwhile a
    lift to $\scHom(\uDelta^1, K_\dt)$ is a map $\Delta^k \times
    \Delta^1 \to K(V)$, which by pre-composition with the inclusion
    $\dDelta^1 \subset \Delta^1$ yields a map $\Delta^k \times
    \dDelta^1 \to K(V)$.  We require these two maps are equal.

    But we can guarantee this as follows: giving the lifting diagram
    above, we extend by projection to the first factor to a diagram
    \[ \xymatrix{\dDelta^k \times \dDelta^1 \ar@{^(->}[d] \ar[r] &
      \dDelta^k \ar[r] \ar@{^(->}[d] & F_\dt(V) \ar[r] & L_\dt(V)
      \ar[d] \\ \Delta^k \times \Delta^1 \ar@{-->}[urrr] \ar[r] &
      \Delta^k \ar[rr] && X_\dt(V).  } \] By Lemma \ref{l:cofib} we
    can lift to get the dashed arrow.  This yields a composition
    \[ \Delta^k \times \dDelta^1 \to \Delta^k \times \Delta^1 \to
    L_\dt(V) \to K_\dt(V) \] e.g. lifts to $L_\dt(V) \times L_\dt(V)$
    and $\scHom(\uDelta^1, K_\dt)(V)$ which map to the same the lift
    to $K_\dt(V) \times K_\dt(V)$.  The compatibility of the maps to
    $X_\dt$, and the fact that $L_\dt \to L_\dt \times L_\dt$ is the
    diagonal, ensures that these lifts are compatible with the maps in
    the fiber product.  
  \end{itemize}
\end{proof}

Fix a $X_\dt \in s\Pre C$.  Let $\Ab(\Sh C)$ be the category of
sheaves of abelian groups on $C$.  A simplicial presheaf $K_\dt$
yields a negative cochain complex of sheaves of free abelian groups
$\Z{K^\dt}$.  We abuse notation and also call $D(X_\dt)$ the image of
$D(X_\dt)$ inside $\Ch^-(\Ab(\Sh C))$ under this functor.  Note that
simplicial homotopy of simplicial presheaves becomes chain homotopy of
cochain complexes under this functor.

Our basic result on hyper-\v{C}ech cohomology is

\begin{thm}\label{t:main}
  Let $X_\dt$ be a simplicial presheaf, and $hD(X_\dt)$ the homotopy
  category of local acyclic fibrations $K_\dt \to X_\dt$.  Then for a
  bounded below complex of sheaves of abelian groups $\cF^\dt$ with
  the filtration b\^{e}te
  \[ \varinjlim_{K_\dt \in hD(X_\dt)} H^p(\Tot\Hom(\Z{K^\dt}, \cF^\dt))
  \simeq \Ext^p(\Z{X^\dt}, \cF^\dt) \] and there is a
  filtered quasi-isomorphism of ind-objects in the derived category
  \[ \underset{K_\dt \in hD(X_\dt)}{"\varinjlim"} \Tot\Hom(\Z{K^\dt},
  \cF^\dt) \simeq R\Hom(\Z{X^\dt}, \cF^\dt). \]
\end{thm}
\begin{proof}
  According to Proposition \ref{p:verd}, $D(X_\dt)$ is a subcategory
  of $\Ch^-(\Ab(\Sh C))$ which satisfies the properties of the Lemma
  \ref{l:ext}, the lemma on computing $\Ext$, which gives the result.
  For the last part, apply Corollary \ref{c:rhom}.
\end{proof}

\subsection{Semi-representability and finite representability}

\begin{defn}
  A presheaf is \emph{semi-representable} if it is isomorphic to a
  coproduct of representable presheaves.  A presheaf is \emph{finitely
    representable} if it is isomorphic to a finite coproduct of
  representable presheaves.  A simplicial presheaf is
  semi-representable (resp. finitely representable) if all its
  components are.
\end{defn}

The theorems above show representability hypotheses are not important
in the computation of sheaf cohomology.  However typically one wishes
to compute with representable or semi-representable presheaves.  For
this we have

\begin{lem}[A Godement-type lemma]\label{l:god} Any presheaf is
  covered by a semi-\\representable presheaf.
\end{lem}
\begin{proof} For a presheaf $F$, we have the presheaf surjection
  \[ \coprod_{(X \in C, s \in F(X))} h_X \to F \] where $h_X$ denotes
  the representable presheaf given by $\Hom(\blank, X)$.  Since
  $\Hom(X, F) = F(X)$ the morphism is given by $s$.  This is obviously
  surjective on the level of sets, and since sheafification is exact,
  is a covering.
\end{proof}

\begin{rem}\label{r:split}
  We make use of the formalism of \emph{split simplicial objects},
  cf. \cite[ex Vbis 5.1]{MR0354653} or \cite[6.2.2]{MR0498552}, which
  allows us to construct semi-representable simplicial presheaves
  inductively by only specifying the non-degenerate pieces.  The
  degeneracies are satisfied by adding copies of the lower degree
  pieces; all maps between such objects are isomorphisms, so will
  satisfy whatever requirements we have of them (properness,
  coverings, et cetera) and will come equipped inductively via the
  degeneracies with maps to any desired target.
\end{rem}

\begin{prop}
  Let $SR(X_\dt)$ be the full subcategory of $D(X_\dt)$ of objects
  whose components are semi-representable.  Then $hSR(X_\dt)$ is a
  cofinal subcategory of $hD(X_\dt)$.  Thus for a bounded below
  complex of sheaves of abelian groups $\cF^\dt$ with the filtration
  b\^{e}te
  \[ \varinjlim_{K_\dt \in hSR(X_\dt)} H^p(\Tot\Hom(\Z{K^\dt},
  \cF^\dt)) \simeq \Ext^p(\Z{X^\dt}, \cF^\dt) \] and there is a
  filtered quasi-isomorphism of ind-objects in the derived category
  \[ \underset{K_\dt \in hSR(X_\dt)}{"\varinjlim"} \Tot\Hom(\Z{K^\dt},
  \cF^\dt) \simeq R\Hom(\Z{X^\dt}, \cF^\dt). \]
\end{prop}
\begin{proof} It is enough to show, for any local acyclic fibration
  $K_\dt \to X_\dt$, there is a local acyclic fibration $L_\dt \to
  K_\dt$ with $L_\dt$ semi-representable.  We construct one
  inductively as follows: set $L_0 \to K_0$ a semi-representable cover
  given by the Godement lemma.  Having constructed $L_\dt$ to degree
  $i-1$, set \[ L' \to (i_{i-1*}L)_i \times_{(\cosk_{i-1}K)_i} K_i\]
  to be a semi-representable cover given by the Godement lemma.  We
  set $L_i$ to be the union of $L'$ and the copies of the $L_k$ for $k
  < i$ needed to satisfy the degeneracy relations; see Remark
  \ref{r:split}.
\end{proof}

\begin{rem}
  This gives a generalized version of Verdier's theorem on hypercovers
  (\cite[ex V Theorem 7.4.1]{MR0354653}.)
\end{rem}

On sites without finite products and fiber products, we need some
additional hypotheses for finite representability.  The following
result will be useful in application to $\Sm_h \subset \Sch_h$,
cf. Corollary \ref{c:rep}.

\begin{prop}\label{p:frep}
  Suppose the site $C$ is full subcategory of a larger site $C'$, and
  \begin{enumerate}
  \item The topology on $C'$ is generated by a pretopology all of
    whose covering families are finite.
  \item $C$ has the induced topology.
  \item $C'$ has finite products and fiber products.
  \item Every $Y \in C'$ can be covered by an $X \in C$.
  \end{enumerate}
  Let $X_\dt \in s\Pre C'$ be a finitely representable simplicial
  presheaf.  Let $FR_C(X_\dt)$ be the subcategory of $SR_{C'}(X_\dt) =
  \set{\text{semi-representable local acyclic fibrations in } C'}$
  whose components are finitely representable and in $C$.  Then
  $hFR_C(X_\dt)$ is cofinal in $hSR_{C'}(X_\dt)$.  Thus for a bounded
  below complex of sheaves of abelian groups $\cF^\dt$ with the
  filtration b\^{e}te
  \[ \varinjlim_{K_\dt \in hFR_C(X_\dt)} H^p(\Tot\Hom(\Z{K^\dt},
  \cF^\dt)) \simeq \Ext^p(\Z{X^\dt}, \cF^\dt) \] and there is a
  filtered quasi-isomorphism of ind-objects in the derived category
  \[ \underset{K_\dt \in hFR_C(X_\dt)}{"\varinjlim"} \Tot\Hom(\Z{K^\dt},
  \cF^\dt) \simeq R\Hom(\Z{X^\dt}, \cF^\dt). \]
\end{prop}
\begin{proof}
  The hypotheses on $C$ and $C'$ show that
  \begin{enumerate}
  \item Every covering morphism $F \to G$ in $C'$ where $F$ is
    semi-representable and $G$ is finitely representable can be
    refined $E \to F \to G$ where $E$ is finitely representable in $C$
    and $E \to G$ is a covering morphism.
  \item Finite limits of finitely representable presheaves in $C'$ can
    be covered by finitely representable presheaves in $C$.
  \end{enumerate}

  By Verdier's theorem it is enough to show, for every
  semi-representable local acyclic fibration $K_\dt \to X_\dt$ in
  $C'$, there is a finitely representable local acyclic fibration
  $L_\dt \to X_\dt$ with $L_\dt$ in $C$ and a map over $X_\dt$ of
  simplicial presheaves $L_\dt \to K_\dt$.  Set $L_0 \subset K_0$ a
  subpresheaf which is finitely representable in $C$ and covers $X_0$.
  Suppose inductively we have constructed $L_\dt$ to degree $i-1$.
  Then $(i_{i-1*}L)_i \times_{(\cosk_{i-1}X)_i} X_i$ is a finite limit
  of finitely representable presheaves, so cover it with $L'$ a
  finitely representable in $C$.  Construct the fiber product $L'
  \times_{(\cosk_{i-1}K)_i} K_i$, cover it with a semi-representable
  $L''$.  Then $L'' \to L'$ is a cover of a finitely representable by
  a semi-representable, so take $L''' \to L'' \to L'$ with $L'''$
  finitely representable in $C$and $L''' \to L'$ a cover.  As before
  we have to add copies of $L_k$ for $k < i$ to satisfy degeneracy
  conditions, cf. Remark \ref{r:split}.  By construction there is a
  map $L_\dt \to K_\dt$ and the composite $L_\dt \to X_\dt$ is a local
  acyclic fibration.
\end{proof}

\section{{$h$}- AND {$ucd$}-TOPOLOGIES}

\subsection{The $h$-topology}

\begin{defn}
  A $\C$-scheme is a separated scheme finite type over the field of
  complex numbers.  Let $\Sch$ denote the category of $\C$-schemes,
  and let $\Sm \subset \Sch$ denote the full subcategory of smooth
  $\C$-schemes.  If $X \in \Sch$, let $\Sch / X, \Sm / X$ denote the
  categories of $\C$-schemes and smooth $\C$-schemes over $X$.
\end{defn}

We recall Voevodsky's (\cite{MR1403354}) $h$-topology:

\begin{defn}
  A morphism $f:X \to Y$ is called a \emph{topological epimorphism} if
  the underlying morphism of topological spaces is a topological
  quotient map: it is surjective on sets and $U \subset Y$ is open if
  and only if $f^\inv(U)$ is open in $X$.  A \emph{universal
    topological epimorphism}, or an $h$-covering, is a morphism $X \to
  Y$ so that for any $Z \to Y$, the base change morphism
  \[ X \times_Z Y \to Z \] is a topological epimorphism. 
\end{defn}

A useful necessary but not sufficient characterization of
$h$-coverings is given by the following.

\begin{prop}[{\cite[Proposition 3.1.3]{MR1403354}}]
  Let $f:X \to Y$ be a morphism of schemes, and $X' \subset X$ the
  union of the irreducible components of $X$ which dominates some
  component of $Y$.  If $f$ is an $h$-covering then $f(X') = Y$.
\end{prop}

\begin{defn}
  The $h$-topology is the topology on $\Sch$ induced from the
  pretopology given by finite families $\set{U_i \to X}$ where
  $\coprod U_i \to X$ is an $h$-covering.  We denote the site of
  $\C$-schemes with the $h$-topology $\Sch_h$.  $\Sm$ inherits a
  topology from $\Sch_h$ as in \cite[ex III 3.1]{MR0354652}; by
  resolution of singularities (\cite{MR0199184,MR1423020}) this
  topology is just given by restricting covering sieves of $\Sch_h$ to
  $\Sm$; we denote this site $\Sm_h$.
\end{defn}

\begin{rem}
  Note that, by resolution of singularities
  (\cite{MR0199184,MR1423020}), $\Sm_h \subset \Sch_h$ satisfy the
  conditions of Proposition \ref{p:frep}, so finitely representable
  hypercovers compute sheaf cohomology in $\Sm_h$.
\end{rem}

\begin{thm}[{\cite[ex III Theorem 4.1]{MR0354652}}]
  Let $C, C'$ be small categories, $u:C \to C'$ a fully faithful
  functor.  Suppose $C'$ has a Grothendieck topology, and let $C$ have
  the induced topology.  If every object of $C'$ can be covered by an
  object of $C$, then the functor $F \mapsto F \circ u$ is an
  equivalence of the category of sheaves on $C'$ with the category of
  sheaves on $C$.
\end{thm}
 
\begin{cor}\label{c:hequiv}
 The category of sheaves on $\Sm_h$ is equivalent via the natural
 embedding to the category of sheaves on $\Sch_h$.
\end{cor}
\begin{proof}
  Resolution of singularities (\cite{MR0199184,MR1423020}) gives
  smooth $h$-coverings of arbitrary $\C$-schemes.
\end{proof}

\subsection{Cohomological Descent}

\begin{defn}[Cohomological Descent]
  An augmented simplicial $\C$-scheme $e:K_\dt \to X$ is a
  \emph{cohomological descent resolution} if the adjunction \[\id_{an}
  \to Re_{an*}e^*_{an}\] is an isomorphism; here we use the analytic
  topology.  According to \cite[ex XVI 4.1]{MR0354654}, if one
  restricts to rational vector spaces, this is the same as requiring
  $\Q_{l,X} \simeq Re_*(\Q_{l, K_\dt})$ in the \'{e}tale topology.
  The morphism $e$ is a \emph{universal cohomological descent
    resolution} (or a $ucd$-resolution) if it is a cohomological
  descent resolution after any base change.

  A morphism of $\C$-schemes $Y \to X$ is of \emph{cohomological
    descent} if $\cosk_0(Y/X) \to X$ is a cohomological descent
  resolution (where $\cosk_0(Y/X)$ is the coskeleton functor in the
  category of schemes over $X$.)  A morphism $Y \to X$ is
  \emph{universally of cohomological descent} (or a $ucd$-cover) if
  every base change is of cohomological descent.
\end{defn}

Some basic results:

\begin{lem}[{\cite[5.3.5]{MR0498552}}]
  A morphism with a local section is a $ucd$-covering.  A proper
  surjection is a $ucd$-covering.
\end{lem}

\begin{lem}[{\cite[5.3.5]{MR0498552}}]\label{l:ucdcomp}
  \begin{enumerate}
  \item The composition of $ucd$-coverings is a $ucd$-covering.
  \item If the composition $X \to Y \overset{f}{\to} Z$ is a
    $ucd$-covering, then $f$ is a $ucd$-covering.
  \end{enumerate}
\end{lem}
\begin{proof}
  See \cite[Theorem 7.5]{conradCD} for a proof.
\end{proof}

According to \cite[5.3.5]{MR0498552} $ucd$-coverings form a
pretopology on $\Sch$.  We deviate from Deligne, however, in taking
the pretopology generated by only finite families $\set{U_i \to X}$
where $\coprod U_i \to X$ is a $ucd$-covering.  (Deligne and Du Bois
in practice use only representable simplicial objects so there is no
difference.)  We denote the topology generated by this pretopology the
universal cohomological descent topology, or the $ucd$-topology.

Let $\Sch_{ucd}$ be the category of $\C$-schemes with the
$ucd$-topology.  Since resolution of singularities are
$ucd$-coverings, by the exact same argument as for the $h$-topology,
the induced topology on $\Sm$ (denoted $\Sm_{ucd}$) is given by
restricting the covering sieves of $\Sch_{ucd}$, and the categories of
sheaves on $\Sch_{ucd}$ and $\Sm_{ucd}$ are equivalent.

\begin{rem}
  Again $\Sch_{ucd}$ and $\Sm_{ucd}$ satisfy the conditions of
  Proposition \ref{p:frep}, so finitely representable hypercovers
  compute sheaf cohomology in $\Sm_{ucd}$.
\end{rem}

The basic, almost circular theorem is

\begin{thm}[{\cite[5.3.5]{MR0498552}}]
  Let $e:K_\dt \to X$ be a hypercover in the topology of universal
  cohomological descent.  Then $f$ is a universal cohomological
  descent resolution.
\end{thm}

\begin{rem}
  Note that both the $h$-topology and the $ucd$-topology refer to an
  underlying topology: the $h$-topology refers to the Zariski
  topology, and the $ucd$-topology refers to the \'{e}tale or analytic
  topologies.
\end{rem}

\subsection{Comparison of the $h$- and $ucd$-topologies} 

\begin{lem}[{\cite[Theorem 3.1.9]{MR1403354}}]
  An $h$-covering $Y \to X$ of an excellent reduced noetherian scheme
  $X$ can be refined $Y' \to Y \to X$ to an $h$-covering of normal
  form: $Y' \to X$ factors as $s \circ f \circ i$ where $i$ is an open
  covering, $f$ is a finite surjective morphism, and $s$ is a blowup
  of a closed subscheme.
\end{lem}

\begin{cor}\label{c:hfactor}
  An $h$-covering $Y \to X$ in $\Sm_h$ can be refined to $Y' \to Y \to
  X$, where $Y' \to X$ factors into $Y' \to Z \to X$, where $Y' \to Z$
  is a Zariski open cover, and $Z \to X$ is proper, and $Y'$ and $Z$
  are smooth.  Moreover, we may assume both $Y'$ and $Z$ are
  quasi-projective.
\end{cor}
\begin{proof}
  $\C$-schemes are excellent.  Factor $Y \to X$ to $Y'' \to Y \to X$
  with $Y'' \to Z' \to X$ where $Y'' \to Z'$ is a Zariski open cover
  and $Z' \to X$ is proper (composition of a finite morphism and a
  blowup.)  Use resolution of singularities to get $Z \to Z' \to X$
  proper, take $Y' = Y'' \times_{Z'} Z$, which will be a Zariski open
  cover of $Z$.  

  To get the last statement, use Chow's lemma \cite[5.6.1,
  5.6.2]{MR0217084} to get $Z' \to Z$ by a projective surjective
  morphism with $Z'$ quasi-projective, and the base change $Y'' = Y'
  \times_Z Z'$ is a Zariski open cover of $Z'$.
\end{proof}

\begin{cor}
  An $h$-covering in $\Sch$ or $\Sm$ is a $ucd$-covering.
\end{cor}
\begin{proof}
  By the lemma or the corollary an $h$-covering $f:Y \to X$ in either
  $\Sch$ or $\Sm$ has a refinement $Y' \overset{g}{\to} Y
  \overset{f}{\to} X$ in either $\Sch$ or $\Sm$ where $f \circ g$
  factors into a composition of morphisms which are universally of
  cohomological descent.  Hence by Lemma \ref{l:ucdcomp} $f$ is
  universally of cohomological descent.
\end{proof}

By the above proposition, we have continuous functors $\Sch_h \to
\Sch_{ucd}$ and $\Sm_h \to \Sm_{ucd}$ (see \cite[ex III Proposition
1.6]{MR0354652}) and thus geometric morphisms of their associated
topoi of sheaves.  We do not know of an example of a $ucd$-covering
which is not an $h$-covering.

\subsection{Representable hypercovers in the $h$- or $ucd$-topologies}

\begin{lem}
  If $L_\dt$ is a finitely representable hypercover in either $\Sm_h,
  \Sch_h, \Sm_{ucd}$ or $\Sch_{ucd}$, then there is a representable
  hypercover $K_\dt$ in the same site and a morphism $L_\dt \to K_\dt$
  so that $\Z{L^\dt}$ is quasi-isomorphic to $\Z{K^\dt}$.
\end{lem}
\begin{proof}
  For a finite family $\set{U_i}$, 
  \[ \coprod\Hom(\blank, U_i) \to \Hom(\blank, \coprod U_i) \] is a
  Zariski cover, so in particular it is an $h$- and $ucd$-cover.  Thus
  they have the same associated sheaves of abelian groups.

  In addition, if \[ \coprod\Hom(\blank, U_i) \to \coprod\Hom(\blank,
  V_j) \] is a morphism of finitely representable presheaves, then
  Yoneda's lemma tells us the identity morphisms $\id_{U_i} \in
  \Hom(U_i, U_i)$ determine the diagonal in the commutative diagram
  \[ \xymatrix{ \coprod\Hom(\blank, U_i) \ar[r] \ar[d] &
    \coprod\Hom(\blank, V_j) \ar[d] \\ \Hom(\blank, \coprod U_i)
    \ar[ur] \ar@{-->}[r] & \Hom(\blank, \coprod V_j) } \] and thus the
  dashed arrow.  Hence every morphism of finitely representable
  presheaves determines a morphism of associated representable
  coproducts (but not vice versa!) and these morphisms are the same on
  passing to associated sheaves.

  Thus given a finitely representable hypercover $L_\dt$ with $L_n =
  \coprod\Hom(\blank, U_{n,i})$, take $K_\dt$ with $K_n = \Hom(\blank,
  \coprod U_{n,i})$ with simplicial morphisms given as above.  It is
  representable and yields the same complex of sheaves of abelian
  groups (it in fact is a local acyclic fibration in the Zariski
  topology, since it locally has sections.)
\end{proof}

\begin{cor}\label{c:rep}
  Let $X_\dt$ be a representable simplicial presheaf in $\Sch_h$.  Let
  $R_{\Sm}(X_\dt)$ be the subcategory of $FR_{\Sm}(X_\dt) = $
  \{finitely representable local acyclic fibrations in $\Sm_h$\} whose
  components are representable. Then every $L_\dt \in FR_{\Sm}(X_\dt)$
  has a quasi-isomorphism $\Z{L^\dt} \to \Z{K^\dt}$ for some $K_\dt
  \in R_{\Sm}(X_\dt)$, so for a bounded below complex of sheaves of
  abelian groups $\cF^\dt$ with the filtration b\^{e}te
  \[ \varinjlim_{K_\dt \in hR_{\Sm}(X_\dt)} H^p(\Tot\Hom(\Z{K^\dt},
  \cF^\dt)) \simeq \Ext^p(\Z{X^\dt}, \cF^\dt) \] and there is a
  filtered quasi-isomorphism of ind-objects in the derived category
  \[ \underset{K_\dt \in hR_{\Sm}(X_\dt)}{"\varinjlim"}
  \Tot\Hom(\Z{K^\dt}, \cF^\dt) \simeq R\Hom(\Z{X^\dt}, \cF^\dt). \]
\end{cor}
\begin{proof}
  Proposition \ref{p:frep} says we may compute using finitely
  representable hypercovers in $\Sm_h$.  The lemma says finitely
  representable hypercovers have associated complexes of sheaves of
  free abelian groups equivalent to those of representable
  hypercovers.
\end{proof}

\section{ALGEBRAIC DE RHAM COMPLEX}

\subsection{$\Omega^q$ is an $h$-sheaf}

For every $q \ge 0$, let $\Omega^q$ denote the presheaf on the site
$\Sm_h$ given by \[ X \mapsto \Gamma(X, \Omega^q_{X/\C}). \]  It is a
presheaf of $\cO$-modules.

\begin{lem}
  If $f:X \to Y$ is a dominant morphism of smooth $\C$-schemes, then
  $\Omega^q(Y) \hookrightarrow \Omega^q(X)$.
\end{lem}
\begin{proof}
  Suppose $\omega \in \Omega^q(Y)$ has $f^*\omega = 0$.  Generic
  smoothness gives a Zariski open dense $U \subset Y, V = f^\inv(U)
  \subset X$ where $f|_V$ is smooth.  Then $f|_V^*$ is injective so we
  see $\omega$ vanishes on an open dense set, so must be zero.
\end{proof}

\begin{prop}
  $\Omega^q$ is a sheaf in $\Sm_h$.
\end{prop}

\begin{proof}
  We must check, for every covering sieve $R$ of $X$, that
  $\Omega^q(R) = \Omega^q(X)$.  We may assume $X$ is irreducible.  It
  is enough to check for $R$ generated by a single $h$-covering
  family, and in fact a single covering $u:Y \to X$: if $\set{U_i \to
    X}$ is a finite covering family, then $\Omega^q(\set{U_i \to X}) =
  \Omega^q(\coprod_j U_j)$ because $\Omega^q$ is already a Zariski
  sheaf, and ${U_i \to \coprod_j U_j}$ is a Zariski covering.  Since
  every $f \in R$ factors through $u$, the $R$-local sections are just
  elements $\omega \in \Omega^q(Y)$ which, for every pair of maps $f,
  g:Z \rightrightarrows Y$ with $uf = ug$, we have $f^*\omega =
  g^*\omega$. 

  We first check the case where $u$ is a smooth morphism.  In this
  case all pairs $f, g$ factor through the smooth $W = Y \times_X Y
  \rightrightarrows Y$, so it is enough to check for $Z = W.$ For $q =
  0$, this is the usual exact sequence of algebras
  \[ 0 \to A \to B \to B \otimes_A B \] where $A \hookrightarrow B$ is
  the injective map coming from a dominant morphism.  For $q=1$ we
  have from the usual exact sequences of differentials the diagram
 
  \[ \xymatrix{
    0 \ar[dr] \\
    0 \ar[r] & \Omega^1(X) \ar[r] \ar[dr] & \Omega^1(Z) \ar[r] &
    \Gamma(Z, \Omega^1_{Z/X}) = p_1^* \Gamma(Y, \Omega^1_{Y/X})
    \oplus p_2^* \Gamma(Y, \Omega^1_{Y/X}) \\
    & & \Omega^1(Y) \ar@<.5ex>[u] \ar@<-.5ex>[u] \ar@<.5ex>[ur]
    \ar@<-.5ex>[ur] 
  }. \]

  Thus $\Omega^1(X) \hookrightarrow \Omega^1(Y)$, and clearly the
  image is contained in the equalizer of the two vertical arrows.
  Conversely, if a form $\omega \in \Omega^1(Y)$ is sent by both
  vertical arrows to $\eta \in \Omega^1(Z)$, then commutativity of the
  right triangle gives that $\omega$ must be sent to the same place by
  the pair of diagonal arrows.  But the only thing in the intersection
  of the image of $p_1^*$ and $p_2^*$ is zero, hence $\eta$ must lift
  to a form in $\Omega^1(X)$, so $\Omega^1(X)$ is precisely the
  equalizer of the vertical arrows.  The cases $q > 0$ follow from
  applying the (exact) wedge product functor.

  For general $u$, the lemma gives $\Omega^1(X) \hookrightarrow
  \Omega^1(Y)$.  The image of $\Omega^1(X)$ is by definition in the
  intersection of all equalizers.  Conversely, suppose $\omega \in
  \Omega^1(Y)$ is in the equalizer of every pair of arrows $f,g: Z
  \rightrightarrows Y \to X$.  Generic smoothness and the case of a
  smooth morphism show that the result is true at the generic point.
  The proposition then follows from the following lemma.
\end{proof}

\begin{lem}
  Suppose $f:Y \to X$ is an $h$-covering of smooth $\C$-schemes with
  $X$ irreducible.  Let $\set{Y_i}$ be the set of components of $Y$
  which dominate $X$.  Then the diagram
  \[ \xymatrix{ \Gamma(X, \Omega_X^q) \ar@{^(->}[r] \ar@{^(->}[d] &
    \bigoplus_i \Gamma(Y_i, \Omega_{Y_i}^q) \ar@{^(->}[d] \\
    \Omega_X^q \otimes_\C k(X) \ar@{^(->}[r] & \bigoplus_i
    \Omega_{Y_i}^q \otimes_\C k(Y_i) } \] is cartesian: if a $q$-form
  $\omega$ on the generic point of $X$ lifts to a $q$-form on the
  generic point of $Y$ that extends to all of $Y$, then $\omega$
  extends to all of $X$.
\end{lem}
\begin{proof}
  By Hartog's theorem\footnote{Or the algebraic version regarding
    normal varieties and codimension $\ge 2$ sets, see
    \cite[II.8.19]{MR0463157}.} we may safely throw out codimension
  $\ge 2$ subsets of $X$.  Hence if $X' \subset X$ is the open set
  where $\omega \in \Omega_X^q \otimes_\C k(X)$ is defined, we may
  assume the complement $D = X - X'$ is a union of finitely many
  smooth divisors (throwing out singular and intersection sets.)  We
  may extend over one divisor at a time, so assume $D$ is a single
  smooth divisor.

  Note its is enough to prove the lemma after replacing $Y$ with any
  subscheme which dominates $X$ so that $E = f^\inv(D)$ is non-empty.
  Throwing out closed subsets we may assume $E$ is a divisor.  Let
  $\phi:E \to D$ be $f$ restricted to $E$.  Generic smoothness gives a
  point $y \in E$ where $\phi$ is smooth over $x = \phi(y) \in D$.  We
  choose a complementary subspace to $m_{E,y}/m_{E,y}^2 \subset
  m_{Y,y} / m_{Y,y}^2$, and lift generators of this subspace to
  equations $g_1, \dots, g_r$ in $\cO_{Y,y}$.  We replace $Y$ with a
  subvariety defined by the $g_i$ in some neighborhood of $y \in Y$
  where the $g_i$ are defined, so we can assume $\dim Y = \dim X$, and
  throwing out $\codim \ge 2$ points of $X$ and closed subsets of $Y$
  we may assume that $Y$ is smooth and connected, $E$ is a smooth
  connected divisor, and $\phi$ is \'{e}tale at $y$.

  The theorem on the dimension of fibers of a morphism (\cite[II ex
  3.22]{MR0463157}) gives the subset of $U \subset X$ where $f$ is
  quasi-finite is open.  The complement $C = X - U$ is at worst
  dimension $\dim X - 1$.  If it is equal to $\dim X - 1$, then its
  preimage is also $\dim X - 1 = \dim Y - 1$, so applying the theorem
  again to components of $C$ we get a dense open set of $C$ where $f$
  is quasi-finite: thus the subset of $X$ where $f$ is not
  quasi-finite is at least codimension $2$ and we may safely throw
  that out, so we may assume $f$ is quasi-finite.

  By Zariski's Main Theorem (\cite[4.4.3]{MR0217086} or
  \cite[III.9.I]{MR1748380}) we have a factorization $Y \subset
  \Sh X \overset{\pi}{\to} X$ where $Y$ is an open immersion
  in the normalization $\Sh X$ of $X$ in $k(Y)$.  Let $E' =
  \Sh X - Y$.  Let $W = \pi(E') - D$.  Since by Hartog's
  theorem we only have to extend across the generic point of $D$, we
  may throw out $W$.  Hence we may assume $\pi(E') \cap D$ is either
  empty or else is all of $D$.  Throwing out more points we may assume
  $E$ and $E'$ are disjoint smooth divisors.  Again we only have to
  extend over the generic point of $D$, so we may assume $X$ and
  $\Sh X$ are affine.  Let $h, h'$ be defining equations for
  $E, E'$; these exist since the the stalk of $f_*\cO_Y$ over $\cO_{X,
    D}$ is a semi-local PID.  We may assume $h|_{E'} = 1 = h'|_E$.

  We have an $\omega \in \Omega^q_{X/\C} \otimes_\C k(X)$ so that
  $f^*\omega$ extends to an $\eta \in \Gamma(Y, \Omega^q_{Y/\C})$.
  Then for some $m$ large enough $h'^m\eta \in \Gamma(\Sh X,
  \Omega^q_{\Sh X/\C})$.  The theory of traces of $q$-forms
  (for example \cite[4.6.7]{MR868864}) gives us a $q$-form on $X$
  \[ \trace(h'^m \eta). \] Away from $D$, we have \[ \trace(h'^m\eta)
  = \trace(h'^m) \omega \] so it is enough to show that $\trace(h'^m)$
  is invertible.  Since we can throw out closed subsets not containing
  $D$, it is enough to show $\trace(h'^m)|_D$ is invertible.  But this
  is just
  \[ e_{E/D} \trace_{E/D}(h'|_E)^m = e_{E/D} \deg(E \to D)^m \] since
  $h'|E = 1$, where $e_{E/D}$ is the ramification.
\end{proof}

\begin{rem}
  We have a complex of sheaves $\Omega^\dt$ on $\Sm_h$ and an
  augmentation
  \[ 0 \to \C \to \cO \to \Omega^1 \to \Omega^2 \to \cdots \] coming
  from the usual inclusions and exterior differentiation.  The complex
  $\Omega^\dt$ has a natural filtration, the filtration b\^{e}te.
\end{rem}

Fix an $X \in \Sch$.  For simplicity we assume $X$ is irreducible.  We
consider the sites $\Sm_h/X$ of smooth $\C$-schemes over $X$,
$\Sch_h/X$ all $\C$-schemes over $X$, and $X_{\Zar}$ the small site of
Zariski-open subsets of $X$.  The natural inclusion $\gamma:X_{\Zar}
\hookrightarrow \Sch_h / X$ gives $X_{\Zar}$ the induced Grothendieck
topology, since a family of Zariski open sets is a Zariski cover only
if it is an $h$-cover.  Therefore $\gamma$ is continuous \cite[ex III
3.1]{MR0354652} and induces a geometric morphism of topoi \cite[ex III
1.2.1]{MR0354652} which we also denote by $\gamma$: \[ \gamma =
(\gamma^*, \gamma_*): \Sh \Sm_h / X \simeq \Sh \Sch_h / X \to \Sh
X_{\Zar} \] the first equivalence being given by Corollary
\ref{c:hequiv}. Perhaps confusingly, for an $h$-sheaf $F$ we have
$\gamma_* F = F \circ \gamma$.  Note that \[ \Z{X_h} = \gamma^*
\Z{X_{\Zar}} \] as both are the sheaf of free abelian groups
associated to the constant presheaf with value $\Z$.

\begin{rem}
  Since $\Omega^q$ is a sheaf on $\Sm_h$, for any $X \in \Sch$ and any
  diagram
  \[ X \leftarrow X_0 \leftleftarrows X_0 \times_X X_0 \leftarrow X_1
  \] where $X_0 \to X$ is an $h$-covering and $X_0, X_1 \in \Sm$,
  $\gamma_*\Omega^q_X$ is determined by the exact sequence
  \[ 0 \to \Gamma(X, \gamma_*\Omega^q_X) \to \Gamma(X_0,
  \Omega^q_{X_0/\C}) \rightrightarrows \Gamma(X_1, \Omega^q_{X_1/\C}).
  \] This shows $\gamma_*\Omega^q_X$ is quasi-coherent.  Since by
  \cite{MR0199184,MR1423020} we can choose proper $h$-covers,
  $\gamma_*\Omega^q_X$ is coherent.
\end{rem}

\subsection{Results of Du Bois}

\begin{defn}
  Let $X$ be a $\C$-scheme.  A \emph{good cover} of $X$ is a smooth
  representable $h$-hypercover $Z_\dt \to X$ with components
  quasi-projective and proper over $X$.
\end{defn}

\begin{thm}[{\cite[3.11]{MR613848}}]\label{t:dubois}
  Let $X$ be a $\C$-scheme, and $e:K_\dt \to X, e':K'_\dt \to X$ two
  good covers of $X$.  Let $\alpha:K'_\dt \to K_\dt$ be a map over
  $X$.  Then the induced map
  \[ Re_*(\Omega^p_{K_\dt/\C}) \to Re'_*(\Omega^p_{K'_\dt/\C}) \] is
  an isomorphism in the derived category.
\end{thm}

The morphism is constructed by applying $Re_*$ to
\[ \Omega^p_{K_\dt/\C} \to \alpha_* \Omega^p_{K'_\dt/\C} \to R\alpha_*
\Omega^p_{K'_\dt/\C}. \] This direct image is computed in the Zariski
topology; by GAGA \cite{MR2017446} this commutes with analytification,
since all components are proper over the base $X$.

\begin{cor}[{\cite[3.17]{MR613848}}]\label{c:dubois}
  Same hypotheses as above.  Giving the complexes
  $\Omega^\cdot_{K_\dt/\C}$, $\Omega^\cdot_{K'_\dt/\C}$ the filtration
  b\^{e}te, the canonical map
  \[ Re_*(\Omega^\cdot_{K_\dt/\C}) \to Re'_*(\Omega^\cdot_{K'_\dt/\C})
  \] is an isomorphism in the filtered derived category.
\end{cor}

For $X$ smooth, we can take $K_\dt = X_\dt$ the constant simplicial
scheme; this clearly is a smooth resolution of $X$.  In this case the
theorems degenerate to

\begin{prop}
  For $X$ a smooth $\C$-scheme and good cover $e':K'_\dt \to X$, we
  have
  \[ e'_* \Omega^q_{K'_\dt/\C} = \Omega^q_{X/\C} \] and $R^ie'_*
  \Omega^q_{K'_\dt/\C} = 0$ for $i > 0$.
\end{prop}

\begin{cor}
  Same hypotheses as above.  Giving the complexes
  $\Omega^\cdot_{K'_\dt/\C}$, $\Omega^\cdot_{X/\C}$ the filtration
  b\^{e}te, the canonical map
  \[ Re_*(\Omega^\cdot_{K'_\dt/\C}) \to \Omega^\cdot_{X/\C}
  \] is an isomorphism in the filtered derived category.
\end{cor}

\subsection{Comparison of $h$- and Zariski topology.}

The following result comes from generous suggestion of Alexander
Beilinson.

\begin{thm}\label{t:beilinson}
  Let $X$ be a $\C$-scheme, and $\cF^\dt$ a bounded below complex of
  sheaves of abelian groups in $\Sm_h/X$ given the filtration
  b\^{e}te.  Let $Q(X)$ be the subcategory of good covers of $X$ in
  $R_{\Sm}(X)$ the category of representable smooth $h$-hypercovers of
  $X$.  Then the associated homotopy category of cochain complexes
  $hQ(X)$ is cofiltered, and there is a filtered quasi-isomorphism of
  ind-objects
  \[ \underset{Z_\dt \in hQ(X)} {"\varinjlim"} R\Hom_{\Zar}(\Z{Z^\dt},
  \cF^\dt)) \simeq R\Hom_h(\Z{X}, \cF^\dt). \] 
\end{thm}
\begin{proof}
  We construct, for any smooth representable $h$-hypercover $K_\dt \to
  X$, a diagram
  \[ Z_\dt \overset{\pi}{\leftarrow} L_\dt \overset{\phi}{\rightarrow}
  K_\dt \] where $L_\dt$ is a smooth representable $h$-hypercover of
  $X$, $Z_\dt$ is a good cover of $X$, and $\pi$ is a local acyclic
  fibration in the Zariski topology.  Assuming such a construction
  exists, then the $L_\dt$ are cofinal in all smooth representable
  hypercovers, so by Corollary \ref{c:rep} we have a filtered
  quasi-isomorphism
  \[ \underset{L^\dt}{"\varinjlim"} \Tot\Hom(\Z{L^\dt}, \cF^\dt) \simeq
  R\Hom_h(\Z{X}, \cF^\dt). \] Note of course $\Tot\Hom(\Z{L^\dt}, \blank)$
  does not see the topology.
  
  Now there is also a natural morphism
  \[ \underset{L^\dt}{"\varinjlim"} \Tot\Hom(\Z{L^\dt}, \cF^\dt) \to
  \underset{Z^\dt}{"\varinjlim"} R\Hom_{\Zar}(\Z{Z^\dt}, \cF^\dt) \]
  where the $Z_\dt$ run over good covers of $X$.  Each $L_\dt$ is a
  Zariski local acyclic fibration of some $Z_\dt$, which gives the
  map.  We claim in the limit this is a filtered quasi-isomorphism.
  This is because we can compute $R\Hom_{\Zar}(\Z{Z^\dt}, \cF^\dt)$ as
  the limit of \v{C}ech cohomology over Zariski local fibrations
  $L_\dt \to Z_\dt$ by Corollary \ref{c:rep}, and every such $L_\dt$
  appears on the left side.  The composition gives the desired
  filtered quasi-isomorphism
  \[ \underset{Z^\dt \in hQ(X)} {"\varinjlim"} R\Hom_{\Zar}(\Z{Z^\dt},
  \cF^\dt)) \simeq R\Hom_h(\Z{X}, \cF^\dt). \] 

  To constuct the $L_\dt$ and the $Z_\dt$, in degree zero we
  form the diagram of smooth $\C$-schemes
  \[ \xymatrix{ Z_0 \ar[dr]_{\text{proper}} & L_0 \ar@{^(->}[l]_{\Zar}
    \ar[r] & K_0 \ar[dl]^{h-\text{cover}} \\ & X & } \] which exists
  by Corollary \ref{c:hfactor}.  (Here ``$\Zar$'' indicates a Zariski
  open cover and ``proper'' indicates a proper surjective cover.)
  Assume inductively we have constructed a diagram of $n$-truncated
  objects
  \[ Z_{\le n} \overset{\pi_{\le n}}{\leftarrow} L_{\le n} \rightarrow
  K_{\le n} \] where
  \begin{enumerate}
  \item All objects are smooth representable and the $K_{\le n}$ is
    the truncation of the $K_\dt$.
  \item $Z_{\le n}$ is an $n$-truncated good cover of $X$.
  \item $L_i \to Z_i$ is a Zariski open cover for all $0 \le i \le n$.
  \end{enumerate}
  
  Note that 3 implies $L_{\le n} \to Z_{\le n}$ is a local
  acyclic fibration: for $k > n$ the condition
  \[ \xymatrix{ \dDelta^k \ar[r] \ar@{^(->}[d] & L_{\le n}(U) \ar[d]
    \\ \Delta^k \ar[r] \ar@{-->}[ru] & Z_{\le n}(U) } \] is empty,
  and for $k \le n$ a local section $Z_k(V) \to L_k(V)$ allows us to
  lift.

  We claim that these assumptions imply $(i_{n*}L_{\le n})_{n+1} \to
  (i_{n*}Z_{\le n})_{n+1}$ is an open Zariski cover.  This comes from
  the following fact: if we have a commutative diagram
  \[ \xymatrix{ & & B' \ar[dd] \ar[dr] \\
    &&& B \ar[dd]  \\
    A' \ar[rr] \ar[dr] & & C' \ar[dr] \\
    & A \ar[rr] & & C } \] where all the diagonal arrows are Zariski
  covers, then $A' \times_{C'} B' \to A \times_C B$ is a Zariski
  cover.  To see this, first we get map a $A' \to A \times_C C'$.
  This is a Zariski open cover because the composite with the
  projection to $A$ is an open cover, and $A \times_C C' \to A$ itself
  is a Zariski open cover.\footnote{It is an \'{e}tale surjective
    monomorphism on components.}  Likewise for $B' \to B \times_C C'$.
  Thus the map
  \[ A' \times_{C'} B' \to (A \times_C C') \times_{C'} (B \times_C C')
  = (A \times_C B) \times_C C' \] is a Zariski open cover.  But $(A
  \times_C B) \times_C C' \to A \times_C B$ is a Zariski open cover by
  base change, hence so is $A' \times_{C'} B' \to A \times_C B$.

  Now the $i_{n*}$ are constructed by finite products and fiber
  products of the $L_k$'s and $Z_k$'s, and the morphism is
  component-by-component, and these are all Zariski open covers.
  Hence repeating the argument above will show that $(i_{n*}L_{\le
    n})_{n+1} \to (i_{n*}Z_{\le n})_{n+1}$ is a Zariski
  cover.

  Now construct the diagram
  \[ \xymatrix{ & C \ar[ld]_{\Zar} \ar[rr] &
    & B \ar[ddl]^h \ar[dr]  \\
    Z'_{n+1} \ar[d]_{\text{proper}} & & L'_{n+1} \ar[ul]^{\Zar}
    \ar[d]^h & & K_{n+1} \ar[d]^h \\
    (i_{n*}Z_{\le n})_{n+1} & & (i_{n*}L_{\le n})_{n+1}
    \ar[ll]^{\Zar} \ar[rr] & & (\cosk_n K)_{n+1} }\] where
  \begin{enumerate}
  \item $\Zar$ indicates an arrow is a Zariski open cover,
    ``proper'' a proper surjective cover, and $h$ an $h$-cover;
  \item $B$ is the fiber product $K_{n+1} \times_{(\cosk_n K)_{n+1}}
    (i_{n*}L_{\le n})_{n+1}$;
  \item $C \overset{\Zar}{\longrightarrow} Z'_{n+1}
    \overset{\text{proper}}{\longrightarrow} (i_{n*}Z_{\le n})_{n+1}$
    is the factorization of the $h$-covering $B \to (i_{n*}Z_{\le
      n})_{n+1}$ given by Corollary \ref{c:hfactor}, so $Z'_{n+1}$ is
    smooth representable proper over $(i_{n*}Z_{\le n})_{n+1}$ with
    quasi-projective components;
  \item $L'_{n+1}$ is the fiber product $C \times_{(i_{n*}Z_{\le
        n})_{n+1}} (i_{n*}L_{\le n})_{n+1}$.  In particular, it is an
    open Zariski cover of $C$, and hence smooth.
  \end{enumerate}
  Then up to degeneracies, the $L'_{n+1}$ and $Z'_{n+1}$ satisfy all
  the conditions needed: the $Z'_{n+1}$ is quasi-projective and proper
  and surjective over the coskeleton and $X$ and $L'_{n+1}$ is a
  Zariski cover of $Z'_{n+1}$ and completes $L_{\le n}$ to a truncated
  $h$-hypercover.  By Remark \ref{r:split} we can fulfill degeneracy
  conditions by adding disjoint unions with lower degree pieces, which
  does not affect any of the properties we have established.

  Finally we have to show the map $L'_{n+1} \to K_{n+1}$ is compatible
  with the face maps, in other words, the direct map $L'_{n+1} \to
  (i_{n*}L_{\le n})_{n+1}$ given by the vertical arrow factors through
  $B$, e.g. is equal to
  \[ L'_{n+1} \to C \to B \to (i_{n*}L_{\le n})_{n+1}. \] This can be
  checked as follows: first we simplify the notation in the diagram
  \[ \xymatrix{ & C \ar[ld] \ar[rr] &
    & B \ar[ddl]   \\
    Z \ar[d] & & L \ar[ul]
    \ar[d] \\
    \overline{Z} & & \overline{L} \ar[ll]^{\Zar} } \] where $Z =
  Z_{n+1}, \overline{Z} = (i_{n*}Z_{\le n})_{n+1}$ and likewise for
  $L$.  We write composites $C \to B \to \overline{L}$ as
  $CB\overline{L}$, et cetera.  Then by construction
  $CB\overline{L}\overline{Z} = CZ\overline{Z}$, and $LCZ\overline{Z}
  = L\overline{L}\overline{Z}$.  Hence $LCB\overline{L}\overline{Z} =
  L\overline{L}\overline{Z}$; but $\overline{L}\overline{Z}$ is an
  epimorphism, so we can right-cancel, yielding $LCB\overline{L} =
  L\overline{L}$, which is what we wanted.  Hence the map $L'_{n+1}
  \to K_{n+1}$ is consistent with face maps.  The degeneracies are
  automatic by the splitting construction.  Thus completes the
  construction and the proof.
\end{proof}

\subsection{Algebraic de Rham complex}

\begin{prop}
  Let $X$ be a $\C$-scheme.  Then for any good cover $Z_\dt$ of
  $X$, we have a filtered quasi-isomorphism
  \[ R\Hom_{\Zar}(\Z{Z^\dt}, \Omega^\dt) \simeq R\Hom_h(\Z{X},
  \Omega^\dt). \]
\end{prop}
\begin{proof}
  By the result of Du Bois (Corollary \ref{c:dubois}), every term in
  the ind-object of Theorem \ref{t:beilinson}
  \[ \underset{Z_\dt \in hQ(X)}{"\varinjlim"} R\Hom_{\Zar}(\Z{Z^\dt},
  \Omega^\dt) \] is isomorphic (recall $Q(X)$ is the category of Du
  Bois covers of $X$.)
\end{proof}

\begin{defn}
  Let $e:Z_\dt \to X$ be a good cover of $X$.  Define the algebraic
  de Rham complex as
  \[ \underline{\Omega^\cdot_X} := Re_* \Omega^\cdot_{K_\dt/\C}. \]
\end{defn}

Recall $\gamma_*:\Sh\Sm_h/X \to \Sh X_{\Zar}$ is the direct image of
sheaves on the smooth $h$-site over $X$ to sheaves on the small
Zariski site of $X$.

\begin{thm}
  The algebraic de Rham complex is quasi-isomorphic in the filtered
  derived category to $R\gamma_* \Omega^\cdot$ with the filtration
  b\^{e}te.
\end{thm}
\begin{proof}
  Apply the previous proposition Zariski locally on the base $X$.
\end{proof}

\begin{cor}
  Algebraic de Rham cohomology, with the filtration b\^{e}te, is
  computed by the hypercohomology of $\Omega^\cdot$ in $\Sm_h / X$
  with the filtration b\^{e}te.
\end{cor}
\begin{proof}
  \[ \HH^i_h(\Sm_h / X, \Omega^\cdot) = \HH^i(X_{\Zar},
  R\gamma_*\Omega^\cdot) = \HH^i(X_{\Zar},
  Re_*\Omega^\cdot_{K_\dt/\C}) = H^i_{dR}(X). \]
\end{proof}

\begin{rem}
  As noted before, this filtration is typically not the Hodge
  filtration.
\end{rem}

\subsection{Questions}

\begin{enumerate}
\item For an open $\C$-scheme $X$, is there a site of ``log $h$-covers
  of $X$'' which takes the place of Deligne's construction of smooth
  hypercovers with boundary a normal crossing divisor?  
\item Is there a model-theoretic generalization of Lemma \ref{l:ext}?
\item What are the minimum hypotheses about $\Omega^q$ which allow the
  Du Bois results to go through?  Is the following enough: $\cF$ is
  sheaf of $\cO$-modules on $\Sm_h$, locally free on smooth Zariski
  sites, with ``transfers?''
\item Is there a difference between the $\Sm_{ucd}$ and $\Sm_h$?
\item Is there a characterization of hypercovers in terms of ordinary
  covers if one works with the geometric realization?
\item The genesis of all of this work was an idea of Nori, on
  ``holomorphic Whitney forms.''  The basic idea was to look at
  functionals on cycles which ``vary holomorphically,'' in analogy
  with \cite{MR0087148}; a discussion will be forthcoming in a future
  article.  What is the relationship between this theory,
  ``holomorphic Whitney forms,'' and intersection cohomology sheaves?
\end{enumerate}

\bibliography{paper}{}

\begin{thebibliography}{10}

\bibitem{conradCD}
Brian Conrad.
\newblock Cohomological descent.
\newblock Available at \\
  \url{http://www.math.lsa.umich.edu/~bdconrad/papers/hypercover.pdf}.

\bibitem{MR1423020}
A.~J. de~Jong.
\newblock Smoothness, semi-stability and alterations.
\newblock {\em Inst. Hautes \'Etudes Sci. Publ. Math.}, (83):51--93, 1996.

\bibitem{MR0498552}
Pierre Deligne.
\newblock Th\'eorie de {H}odge. {III}.
\newblock {\em Inst. Hautes \'Etudes Sci. Publ. Math.}, 44:5--77, 1974.

\bibitem{MR613848}
Philippe Du~Bois.
\newblock Complexe de de {R}ham filtr\'e d'une vari\'et\'e singuli\`ere.
\newblock {\em Bull. Soc. Math. France}, 109(1):41--81, 1981.

\bibitem{MR2034012}
Daniel Dugger, Sharon Hollander, and Daniel~C. Isaksen.
\newblock Hypercovers and simplicial presheaves.
\newblock {\em Math. Proc. Cambridge Philos. Soc.}, 136(1):9--51, 2004.

\bibitem{MR2066498}
Daniel Dugger and Daniel~C. Isaksen.
\newblock Weak equivalences of simplicial presheaves.
\newblock In {\em Homotopy theory: relations with algebraic geometry, group
  cohomology, and algebraic $K$-theory}, volume 346 of {\em Contemp. Math.},
  pages 97--113. Amer. Math. Soc., Providence, RI, 2004.

\bibitem{MR1950475}
Sergei~I. Gelfand and Yuri~I. Manin.
\newblock {\em Methods of homological algebra}.
\newblock Springer Monographs in Mathematics. Springer-Verlag, Berlin, second
  edition, 2003.

\bibitem{MR0217084}
A.~Grothendieck.
\newblock \'{E}l\'ements de g\'eom\'etrie alg\'ebrique. {II}. \'{E}tude globale
  \'el\'ementaire de quelques classes de morphismes.
\newblock {\em Inst. Hautes \'Etudes Sci. Publ. Math.}, (8):222, 1961.

\bibitem{MR0217086}
A.~Grothendieck.
\newblock \'{E}l\'ements de g\'eom\'etrie alg\'ebrique. {IV}. \'{E}tude locale
  des sch\'emas et des morphismes de sch\'emas. {III}.
\newblock {\em Inst. Hautes \'Etudes Sci. Publ. Math.}, (28):255, 1966.

\bibitem{MR0199194}
A.~Grothendieck.
\newblock On the de {R}ham cohomology of algebraic varieties.
\newblock {\em Inst. Hautes \'Etudes Sci. Publ. Math.}, 29:95--103, 1966.

\bibitem{MR0354652}
A.~Grothendieck, editor.
\newblock {\em Th\'eorie des topos et cohomologie \'etale des sch\'emas. {T}ome
  1: {T}h\'eorie des topos}.
\newblock Springer-Verlag, Berlin, 1972.
\newblock S\'eminaire de G\'eom\'etrie Alg\'ebrique du Bois-Marie 1963--1964
  (SGA 4), Dirig\'e par M. Artin, A. Grothendieck, et J. L. Verdier. Avec la
  collaboration de N. Bourbaki, P. Deligne et B. Saint-Donat, Lecture Notes in
  Mathematics, Vol. 269.

\bibitem{MR0354653}
A.~Grothendieck, editor.
\newblock {\em Th\'eorie des topos et cohomologie \'etale des sch\'emas. {T}ome
  2}.
\newblock Springer-Verlag, Berlin, 1972.
\newblock S\'eminaire de G\'eom\'etrie Alg\'ebrique du Bois-Marie 1963--1964
  (SGA 4), Dirig\'e par M. Artin, A. Grothendieck et J. L. Verdier. Avec la
  collaboration de N. Bourbaki, P. Deligne et B. Saint-Donat, Lecture Notes in
  Mathematics, Vol. 270.

\bibitem{MR0354654}
A.~Grothendieck, editor.
\newblock {\em Th\'eorie des topos et cohomologie \'etale des sch\'emas. {T}ome
  3}.
\newblock Springer-Verlag, Berlin, 1973.
\newblock S\'eminaire de G\'eom\'etrie Alg\'ebrique du Bois-Marie 1963--1964
  (SGA 4), Dirig\'e par M. Artin, A. Grothendieck et J. L. Verdier. Avec la
  collaboration de P. Deligne et B. Saint-Donat, Lecture Notes in Mathematics,
  Vol. 305.

\bibitem{MR2017446}
A.~Grothendieck, editor.
\newblock {\em Rev\^etements \'etales et groupe fondamental ({SGA} 1)}.
\newblock Documents Math\'ematiques (Paris) [Mathematical Documents (Paris)],
  3. Soci\'et\'e Math\'ematique de France, Paris, 2003.
\newblock S\'eminaire de g\'eom\'etrie alg\'ebrique du Bois Marie 1960--61.
  [Algebraic Geometry Seminar of Bois Marie 1960-61], Directed by A.
  Grothendieck, With two papers by M. Raynaud, Updated and annotated reprint of
  the 1971 original [Lecture Notes in Math., 224, Springer, Berlin; MR0354651
  (50 \#7129)].

\bibitem{MR0463157}
Robin Hartshorne.
\newblock {\em Algebraic geometry}.
\newblock Springer-Verlag, New York, 1977.
\newblock Graduate Texts in Mathematics, No. 52.

\bibitem{MR0199184}
Heisuke Hironaka.
\newblock Resolution of singularities of an algebraic variety over a field of
  characteristic zero. {I}, {II}.
\newblock {\em Ann. of Math. (2) 79 (1964), 109--203; ibid. (2)}, 79:205--326,
  1964.

\bibitem{MR906403}
J.~F. Jardine.
\newblock Simplicial presheaves.
\newblock {\em J. Pure Appl. Algebra}, 47(1):35--87, 1987.

\bibitem{MR1299726}
Masaki Kashiwara and Pierre Schapira.
\newblock {\em Sheaves on manifolds}, volume 292 of {\em Grundlehren der
  Mathematischen Wissenschaften [Fundamental Principles of Mathematical
  Sciences]}.
\newblock Springer-Verlag, Berlin, 1994.
\newblock With a chapter in French by Christian Houzel, Corrected reprint of
  the 1990 original.

\bibitem{MR868864}
Joseph Lipman.
\newblock {\em Residues and traces of differential forms via {H}ochschild
  homology}, volume~61 of {\em Contemporary Mathematics}.
\newblock American Mathematical Society, Providence, RI, 1987.

\bibitem{MR1748380}
David Mumford.
\newblock {\em The red book of varieties and schemes}, volume 1358 of {\em
  Lecture Notes in Mathematics}.
\newblock Springer-Verlag, Berlin, expanded edition, 1999.
\newblock Includes the Michigan lectures (1974) on curves and their Jacobians,
  With contributions by Enrico Arbarello.

\bibitem{MR1764199}
Andrei Suslin and Vladimir Voevodsky.
\newblock Relative cycles and {C}how sheaves.
\newblock In {\em Cycles, transfers, and motivic homology theories}, volume 143
  of {\em Ann. of Math. Stud.}, pages 10--86. Princeton Univ. Press, Princeton,
  NJ, 2000.

\bibitem{MR1403354}
V.~Voevodsky.
\newblock Homology of schemes.
\newblock {\em Selecta Math. (N.S.)}, 2(1):111--153, 1996.

\bibitem{MR0087148}
Hassler Whitney.
\newblock {\em Geometric integration theory}.
\newblock Princeton University Press, Princeton, N. J., 1957.

\end{thebibliography}
\bibliographystyle{plain}

\end{document}